\documentclass[runningheads,a4paper]{llncs}
\usepackage{amssymb}
\usepackage{graphicx}
\usepackage{times}
\usepackage{xcolor}

\usepackage{url}
\urldef{\mailsa}\path|armario@us.es|
\urldef{\mailsb}\path|dane.flannery@nuigalway.ie|

\newcommand{\mod}{\; \mbox{mod}\;}
\newcommand{\scst}{\scriptscriptstyle}

\begin{document}

\mainmatter  
\title{Generalized binary 
arrays from quasi-orthogonal cocycles}

\author{J. A. Armario\and D.~L.~Flannery}

\institute{Departamento de Matem\'atica Aplicada I, 
Universidad de Sevilla,
Avda. Reina Mercedes s/n, 
41012 Sevilla, Spain\\
\mailsa 
\and 
School of Mathematics, Statistics and Applied 
Mathematics, 
National University of Ireland~Galway, 
Galway H91TK33, 
Ireland\\ 
\mailsb}
\maketitle

\begin{abstract}
Generalized perfect binary arrays (GPBAs) were used by
Jedwab to construct perfect binary arrays. 
A non-trivial GPBA can exist only if its energy is 
$2$ or a multiple of $4$.
This paper introduces \textit{generalized optimal 
binary arrays} (GOBAs) with even energy not divisible 
by $4$, as analogs of GPBAs. 
We give a procedure to construct GOBAs based on a 
characterization of the arrays in terms of 
$2$-cocycles. As a further application, we 
determine negaperiodic Golay pairs arising from  
generalized optimal binary sequences
of small length.
\end{abstract}

\section{Introduction}\label{Introduction}
Let $\phi = (\phi(0), \ldots, \phi(n-1))\in 
\{ \pm 1\}^n$ 
be a binary sequence of length $n$. 
Reading arguments modulo $n$,  
\[
R_\phi(w):=\sum_{k=0}^{n-1} \phi(k)\phi(k+w)
\]
is the {\em periodic autocorrelation of $\phi$ at 
shift $w$}. The {\em expansion}  of $\phi$, denoted 
$\phi'$, is the concatenation of $\phi$ and $-\phi$ 
(in that order).
A pair $\phi_1$, $\phi_2$  of binary sequences, each 
of length $2t$,  such that
$R_{\phi_1'}(w)+R_{\phi_2'}(w)=0$
for $1\leq w \leq 2t-1$ 
(equivalently, for $1\leq
\allowbreak w\leq 4t-1$ and $w\not= 2t$),
is a {\em negaperiodic Golay pair} (NGP).
Note that the original definition of NGP in  
\cite{BD15} coincides with the definition 
above by \cite[Lemma~2]{Ega16}.

We seek good sources of NGPs. 
This objective is connected to several existence 
problems in algebraic design theory. For example,
Egan showed that NGPs of length $2t$ are 
equivalent to certain relative 
$(4t,2,4t,2t)$-difference sets in the 
dicyclic group $Q_{8t}$ of order 
$8t$~\cite[Theorem~3]{Ega16}. 
Actually, there is a relative 
$(4t,2,4t,2t)$-difference set in a central 
extension $E$ of $\mathbb{Z}_2$ by a 
group $G$ of order $4t$, relative to $\mathbb{Z}_2$, 
if and only if there is a Hadamard matrix of 
order $4t$ whose expanded (group-divisible) 
design admits a special regular action by $E$: 
a \textit{cocyclic} Hadamard 
matrix over $G$~\cite[Theorem~2.4]{DFH00}.
By way of \cite[Theorem~3.3]{Fla97}, 
Ito~\cite[p.~370]{Ito97} conjectured that 
$Q_{8t}$ contains such relative difference 
sets for all $t$. Schmidt~\cite{Sch99} 
has verified Ito's conjecture up to $t=46$.
Our recent paper~\cite{AF17} initiated the 
study of {\em quasi-orthogonal} cocycles
over groups $G$ of even order not divisible by 
$4$, in direct analogy with  
cocyclic Hadamard matrices. The present 
paper builds on \cite{AF17}.

It is easy to see that
\begin{equation}\label{boundautocorr}
 \max_{0<w<n}|R_\phi(w)|\geq 
\left\{\begin{array}{cl}
0 & \hspace{10pt} n\equiv  0 \ \mod 4\\
1 & \hspace{10pt} n \equiv 1 \, 
\mbox{ or } \, 3 \ \mod 4\\
2 & \hspace{10pt} n \equiv 2 \ \mod 4.
\end{array}\right.
\end{equation}
The sequence $\phi$ is {\em optimal} if equality 
holds in (\ref{boundautocorr}).
In particular, $\phi$ is {\em perfect} if
$R_\phi(w) =0$ for $0<w<n$.
No perfect binary sequence of length $n>4$ is known.
Attention consequently turns to the larger 
class of {\em perfect binary arrays} (PBAs).
Jedwab~\cite{Jed92} introduced 
\emph{generalized perfect binary arrays} 
(GPBAs) to aid in the construction of PBAs. 
Hughes~\cite{Hug00} subsequently demonstrated the 
cocyclic nature of GPBAs.

A \emph{generalized perfect binary sequence} (GPBS)
is a $1$-dimensional GPBA; 
such $\phi$ have $R_{\phi'}(w)=0$ for all $w$.
Each pair of  GPBSs is obviously an NGP. However, 
a GPBS exists only if $n=2$~\cite[Result~4.8]{Jed92}. 
So let $n>2$ be even; since $R_{\phi'}(w)$ is 
divisible by $4$, and not every $R_{\phi'}(w)$ 
is $0$, some $|R_{\phi'}(w)|$ must be at least $4$.
Thus, we will say that $\phi$ of length $2t$ is a
{\em generalized optimal binary sequence} (GOBS)
if $\max_{0<w<2t}|R_{\phi'}(w)|=4$.
Equivalently, $\phi$ is a GOBS if, for $0<w<2t$,
\[
|R_{\phi'}(w)|=\left\{\begin{array}{cl}
0 & \,\,\mbox{$w$  odd}\\
4 & \,\,\mbox{$w$ even} 
\end{array}\right.
\]
\noindent when $t$ is odd, and
\[
|R_{\phi'}(w)|=\left\{\begin{array}{cl}
4 & \,\,\mbox{$w$ odd}\\
0 & \,\,\mbox{$w$ even} 
\end{array}\right.
\]
when $t$ is even. 
We propose searching for NGPs in the set of 
GOBs of length $2t$, $t$ odd. 

Just as the notion of GPBA extends that of GPBS 
to dimensions greater than $1$, a 
GOBA (\textit{generalized optimal binary array}) is 
a higher-dimensional version of a GOBS.
Section~\ref{Section3} treats GPBAs and GOBAs
from the perspective of \cite{AF17}. 
We prove a one-to-one correspondence between 
GOBAs, quasi-orthogonal cocycles over abelian 
groups, and abelian relative quasi-difference sets.
In Section~\ref{Section4}, we outline and apply 
a method to find NGPs among GOBSs that correspond 
to quasi-orthogonal cocycles over cyclic groups.
The concluding Section~\ref{NormalMatrices}
looks at an important question for cocyclic designs 
prompted by the analysis in Section~\ref{Section4}.

\section{Quasi-orthogonal cocycles and related 
combinatorial structures}
\label{Section2}

Let $G$ and $U$ be finite groups, with $U$ abelian.
A map $\psi : G\times G\rightarrow 
U$ such that $\psi(1,1)= 1$ and
\begin{equation}\label{condiciondecociclo}
\psi(g,h)\psi(gh,k)=\psi(g,hk)\psi(h,k)\quad\ \forall 
\hspace{1pt} g,h,k\in G
\end{equation}
is a (normalized) \emph{cocycle} over $G$. 
If $\phi : G\rightarrow U$ is any map that is 
normalized (i.e., $\phi(1)=1$) then 
$\partial\phi(g,h)=\phi(g)^{-1}\phi(h)^{-1}\phi(gh)$ 
defines a cocycle
$\partial \phi$, called a  {\em coboundary}.
The set of all cocycles over $G$ forms an abelian 
group $Z^2(G,U)$, whose quotient by the subgroup 
$B^2(G,U)$ of coboundaries 
is the \emph{second cohomology group} $H^2(G,U)$. 
We display $\psi\in Z^2(G,U)$ as a \textit{cocyclic 
matrix} $M_\psi = [\psi(g,h)]_{g,h\in G}$. If
$U = \mathbb{Z}_2 = \langle -1 \rangle$ and
$M_\psi$ is Hadamard then $\psi$ is said to 
be \textit{orthogonal}.

The {\it row excess} $RE(M)$ of a cocyclic matrix $M$
indexed by $G$ is 
the sum of the absolute values of 
all row sums, apart from
row $1_G$.
The cocycle equation 
(\ref{condiciondecociclo}) guarantees that
$\psi$ is orthogonal if and only if $RE(M_\psi)$ 
is optimal, i.e., zero.

For the rest of this section, $|G|= 4t+2> 2$.
\begin{proposition}\label{OptimalRowExcess} 
\begin{itemize}
\item[{\rm (i)}]
If $\psi\in Z^2(G,\mathbb{Z}_2)$ then 
$RE(M_\psi)\geq 4t$.
\item[{\rm (ii)}] 
If $\psi\in B^2(G,\mathbb{Z}_2)$ then 
$RE(M_\psi)\geq 8t+2$.
\end{itemize}
\end{proposition}
\begin{proof}
See \cite[Proposition~1]{AF17}. \hfill $\Box$
\end{proof}

In analogy with the definition of 
orthogonal cocycles, we say 
that $\psi$ is {\em quasi-orthogonal}
if its matrix has least possible row excess:  
by Proposition~\ref{OptimalRowExcess}, either 
$\psi\not \in B^2(G,\mathbb{Z}_2)$ 
and $RE(M_\psi)=4t$, or 
$\psi\in B^2(G,\mathbb{Z}_2)$ and
$RE(M_\psi)=8t+2$
(coboundaries were excluded 
from the notion of quasi-orthogonality
in \cite{AF17}).
\begin{lemma}
\label{lemmaquasiortho}
Let $X_m=\{g\in G \ | \
{\textstyle \sum}_{h\in G}\psi(g,h)=
\allowbreak m\}$.
Then $\psi$ is \mbox{quasi-orthogonal} if 
and only if 
$|X_{2}\cup X_{-2}|=4t+1$ for 
$\psi\in B^2(G,\mathbb{Z}_2)$, or
$|X_0|=\allowbreak 2t+1$ and 
$|X_{2}\cup X_{-2}|=2t$ for 
$\psi\not \in B^2(G,\mathbb{Z}_2)$.
\end{lemma}
\begin{proof}
See \cite[Lemma~2.4]{AF17}. \hfill $\Box$
\end{proof}

It is not known whether quasi-orthogonal 
cocycles always exist.
Indeed, we do not know of a group $G$ such that 
$Z^2(G,\mathbb{Z}_2)$ does not
contain a quasi-orthogonal element (in contrast,
there are several non-existence results for
orthogonal cocycles, e.g.,
due to Ito~\cite{Ito94}). 
We have found quasi-orthogonal 
coboundaries over many abelian $G$, but none over 
non-abelian $G$ such as dihedral groups, apart 
from the dihedral group of order $6$. 
Thirdly, for all $t$ such that $4t+1$ is a sum 
of two squares that we tested, we always 
found a quasi-orthogonal cocycle 
$\psi$ over some group of order $4t+2$ 
with $|\mathrm{det}(M_\psi)|$ attaining
the maximum $2(4t+1)(4t)^{2t}$ established 
by Ehlich-Wojtas.
These existence questions
all merit deeper investigation.

Let $E$ be a group with a normal subgroup $N$ of 
order $m$ and index $v$.  
A \emph{relative $(v,m,k,\lambda)$-difference set in 
$E$ relative to $N$} (the \textit{forbidden subgroup})
is a $k$-subset $R$ of 
a transversal for $N$ in $E$ such that
\[
|R\cap xR|=\lambda \quad  \forall \hspace{1pt}
x\in E\setminus N.
\]
Relative $(2s,2,2s,s)$-difference sets are
especially interesting. If $s$ is even then
they are equivalent to cocyclic Hadamard 
matrices~\cite[Corollary~2.5]{DFH00},
whereas none exist if $s$ is odd~\cite{Hir03}. 
In the latter case there is a natural analog of 
relative difference set.
Suppose that $|E|= 8t+4$ and let
$Z\cong \mathbb{Z}_2$ be a normal 
(hence central) subgroup of $E$. A  
{\em relative $(4t+2,2,4t+2,2t+1)$-quasi-difference 
set in $E$ with forbidden subgroup $Z$} is a 
transversal $R$ for $Z$ in $E$ containing a subset 
$S\subset R\setminus \{1\}$ of size $0$ or $2t+1$
such that, for all $x\in E\setminus Z$,
\[
|R\cap xR| = 
\left\{ 
\begin{array}{ll}
2t+1 & \hspace{10pt}  x\in  SZ\\
2t \ \mathrm{or} \ 2t+2 & 
\hspace{10pt} \mathrm{otherwise.} 
\end{array}
\right.
\]
We call $R$ {\em extremal} if $S=\emptyset$.
(This modifies the original definition in 
\cite{AF17} of relative quasi-difference set, 
to allow quasi-orthogonal 
coboundaries).

The next result is mostly Proposition~4.3 in 
\cite{AF17}. For each $\psi \in Z^2(G,{\mathbb Z}_2)$
we have a canonical central extension
$E_\psi$ with element set $\{ (\pm 1,g) \, |\, g\in G\}$
and multiplication defined by 
$(u,g)(v,h)=(uv\hspace{.5pt} \psi(g,h),gh)$.
\begin{proposition}\label{quasiort-quasidiff}
The cocycle $\psi$ is quasi-orthogonal if 
and only if $D=\{(1,g) \ | \ g\in G\}$ is a 
relative $(4t+2,2,4t+2,2t+1)$-quasi-difference 
set in $E_\psi$ with forbidden subgroup 
$\langle (-1,1)\rangle$, where $D$ is extremal 
for $\psi \in B^2(G,\mathbb{Z}_2)$.
\end{proposition}
\begin{remark}
The requisite subset $S$ of $D$ corresponds 
to the rows of $M_\psi$ with  zero sum.
\end{remark}

\section{Generalized binary arrays with optimal autocorrelation}
\label{Section3}

Jedwab~\cite{Jed92} showed that a GPBA is
equivalent to an abelian relative difference set,
and Hughes~\cite{Hug00} identified its underlying 
orthogonal cocycle.
In this section we carry over these ideas 
into the setting of quasi-orthogonal cocycles. 
 
We start with an adaptation of some material 
from \cite{Hug00} and \cite{Jed92}. 
The cyclic group of order $m$ will be written 
additively, i.e.,
as $\mathbb Z_m = \{ 0, 1, \ldots , \allowbreak m-1\}$ 
under addition modulo $m$.
Let ${\bf s}= (s_1,\ldots,s_r)$ be an $r$-tuple
of positive integers greater than $1$,
and let $G={\mathbb{Z}}_{s_1}\times\cdots\times 
{\mathbb{Z}}_{s_r}$. 
A {\em binary ${\bf s}$-array}
is just a set map $\phi : G\rightarrow \{\pm 1\}$;
it has {\em energy} $n:=\prod_{i=1}^r s_i=|G|$.
We view a binary sequence as an $\bf s$-array 
with $r=1$.

Given $\bf s$ and a {\em type vector} 
${\bf z}=(z_1,\ldots,z_r)\in \{ 0,1\}^r$, let
$E= {\mathbb{Z}}_{(z_1+1)s_1}\times
\cdots\allowbreak \times {\mathbb{Z}}_{(z_r+1)s_r}$.
Then
\[
\begin{array}{l}
H=\{h\in E\mid  
h_i=0 \ \mbox{if}\ z_i=0,\  \mbox{and} 
\ h_i=0\ \mbox{or} \ s_i\ \mbox{if}\ z_i=1\}, \\
K=\{ k\in H\mid   k\,\mbox{ has even weight}\}
\end{array}
\]
are elementary abelian $2$-subgroups of $E$. 
Note that $E$ is a (central) extension of $H$ 
by $G$.
For ${\bf z}\neq {\bf 0}$ 
we obtain the short exact sequence
\begin{equation}
\label{sejed}
1 \longrightarrow \langle -1\rangle 
\stackrel{\iota}{\longrightarrow}
E/{K} \stackrel{\beta}{\longrightarrow} 
G \longrightarrow 0 ,
\end{equation}
where $\iota$ maps $-1$ to the generator 
of $H/K$ and $\beta(g+K) = g\, \mod\,{\bf s}$.
This sequence determines a cocycle
$f_{\bf z}\in Z^2(G,\langle -1\rangle)$ after 
choice of a transversal map
$\tau :\allowbreak  G\rightarrow \allowbreak E/K$.  
Specifically, set $\tau(x) = x+K$; then
\[
f_{\bf z}(x, y)= \iota^{-1}(\tau(x)+ \tau(y) 
- \tau(x+y)).
\] 
We can express $f_{\bf z}$ as a product of 
cocycles on cyclic groups.  
Define $\gamma_m\in Z^2(\mathbb{Z}_m,\langle-1\rangle)$
by $\gamma_m(j,k)= (-1)^{\lfloor (j+k)/m\rfloor}$, 
evaluating the exponent as an ordinary integer.
\begin{proposition}[{\cite[Lemma~3.1]{Hug00}}] 
\label{Proposition3}
\begin{itemize}
\item[{\rm (i)}] $f_{\bf z}(x,y)=\prod_{z_i=1}
\gamma_{s_i}(x_i,y_i)$.
\item[{\rm (ii)}] 
$f_{\bf z}\in B^2(G,\langle -1\rangle)$ if and 
only if $s_i$ is odd for all $i$ such that $z_i=1$.
\end{itemize}
\end{proposition}

Each cocycle $\psi\in Z^2(G, \langle -1\rangle)$  
has an associated short exact sequence
\begin{equation}\label{StdSES}
1\longrightarrow \langle -1\rangle 
\stackrel{\iota'}{\longrightarrow} E_{\psi} 
\stackrel{\beta'}{\longrightarrow} G 
\longrightarrow 0,
\end{equation}
where 
$\iota'(u)=(u,0)$ and $\beta'(u,x)=x$.
The following is standard.
\begin{proposition}\label{GammaRef}
If  $\psi$ and $f_{\bf z}$ are cohomologous, 
say $\psi=f_{\bf z}\partial\phi$, then 
{\em (\ref{sejed})} and {\em (\ref{StdSES})} 
are equivalent short exact
sequences: the isomorphism $\Gamma$ defined by
$(u,x) \mapsto \iota(u\phi(x))+\tau(x)$ makes the 
diagram
\[
\begin{array}{ccccccccc}
1 & \longrightarrow & \langle -1\rangle 
& \stackrel{\iota'}{\longrightarrow} & 
E_{\psi} & \stackrel{\beta'}{\longrightarrow} & G 
& \longrightarrow & 0\\[2mm]
& & \| & & 
{\tiny \mbox{$\Gamma$}} 
\big\downarrow\phantom{\Gamma} & & \|
 & &  \\[2mm]
1 & \longrightarrow & 
\langle -1\rangle 
& \stackrel{\iota}{\longrightarrow} & E/K 
& \stackrel{\beta}{\longrightarrow} & G & 
\longrightarrow & 0
\end{array}
\]
commute.
\end{proposition}

We broaden concepts defined earlier only 
for sequences. The {\em expansion}
of a \mbox{binary} $\bf s$-array $\phi$  
with respect to a type vector ${\bf z}$ 
is the map $\phi'$ on $E$ given by
\[
\phi'(g)=\left\{ \begin{array}{rl}
\phi(a) & \quad g\in a + K\\
-\phi(a) & \quad g\notin a +K
\end{array}\right.
\]
where $a$ denotes $g$ modulo ${\bf s}$. 
For any array $\varphi: A\rightarrow \{\pm 1\}$ 
and $x\in A$, let $R_{\varphi}(x)=
\sum_{a\in A} \varphi(a)\varphi(a+x)$.  
\begin{lemma}\label{PhidashFacts}
If $h\in H\setminus K$ then 
$\phi'(h+g) = -\phi'(g)$, 
and if $h\in K$ then 
$\phi'(h+g) = \allowbreak \phi'(g)$.
\end{lemma}
\begin{corollary}\label{RphiNonZeroSplit}
$R_{\phi'}(g)
= |H| \sum_{x\in T} \phi'(x)\phi'(x+g)$
where $T$ is any transversal for $H$ in $E$.
\end{corollary}
\begin{lemma}
\label{GammaTransmit}
The isomorphism $\Gamma$ in 
Proposition{\em ~\ref{GammaRef}} maps 
$\{(1,x)\ | \ x\in G\}\subseteq E_{\psi}$ 
onto $\{g+K\in E/K \mid \phi'(g)=1\}$.
\end{lemma}
\begin{proof}
(Cf.~\cite[p.~330]{Hug00}.)
Let $\phi'(g)=1$ and write $a$ 
for $g$ modulo ${\bf s}$;  then
$g+K=\allowbreak \iota(\phi(a))+a+K=\Gamma((1,a))$.
Conversely, $\Gamma((1,x))=h + x+K$ where $h +K$ is 
the generator of $H/K$ if $\phi(x)=-1$ and 
$h=0$ otherwise. By Lemma~\ref{PhidashFacts}, 
$\phi'(h + x)=\allowbreak 1$. \hfill $\Box$
\end{proof}

The $\bf s$-array $\phi$ is a {\em GPBA$({\bf s})$ 
of type ${\bf z}$} if  
\[
R_{\phi'}(g)=0 \quad 
\forall \hspace{1pt} g \in E\setminus H.
\]
When ${\bf z}=0$, this condition becomes 
(by Corollary~\ref{RphiNonZeroSplit})
\[
R_{\phi}( g )
=0 \quad  \forall \hspace{1pt}
g \in G\setminus \{ 0\}.
\]
In the latter event $\phi$ is a PBA; which is 
equivalent to $\partial \phi$ being orthogonal
(we return to this case later in the section).
More generally, a GPBA$({\bf s})$ 
is equivalent to a relative difference set in 
$E/ K$ relative to $H/K$, hence
equivalent also to a  
cocyclic Hadamard matrix over $G$:
see \cite[Theorem~5.3]{Hug00} 
and \cite[Theorem~3.2]{Jed92}.
So a GPBA can exist only if its energy $n$ is
$2$ or a multiple of $4$. 
Theorems~\ref{th-goba-rqds} and \ref{qc-goba} 
below are analogous results for 
$n\equiv 2 \mod 4$.

Assume that $|G|=4t+2>2$ unless stated otherwise.
Let $s_1/2, s_2, \ldots, s_r$ be odd.
Thus, if $z_1=0$ then $E$ splits over $H$
by Proposition~\ref{Proposition3}, and so
$R_{\phi'}$ is never zero by
Corollary~\ref{RphiNonZeroSplit}
and Lemma~\ref{PhidashFacts}.
\begin{definition}\label{defgoba}
{\em  A {\em GOBA$({\bf s})$ 
of type ${\bf z}$} is a 
binary $\bf s$-array $\phi$ such that}
\begin{itemize}
\item[{\em (i)}]
$R_{\phi'}(g)\in
\{ 0, \pm 2 |H|\}$ $\ \, \forall 
\hspace{1pt} g \in E\setminus H$,
\end{itemize}
{\em and if $z_1 =1$ then}
\begin{itemize}
\item[{\em (ii)}] 
$|\{g\in E \ | \ R_{\phi'}(g) = 0 \}|=|E|/2$.
\end{itemize}
\end{definition}

A GOBS as defined in Section~\ref{Introduction} 
is a GOBA($\bf s)$ with $r=z_1=1$.
When ${\bf z}={\bf 0}$, Definition~\ref{defgoba}
reduces to
\[
R_\phi(g) =\pm 2 \quad \forall 
\hspace{1pt} g \in G\setminus \{ 0\};
\]
we call $\phi$ satisfying this condition
an \textit{optimal binary array} (OBA). 
\begin{lemma}[{\cite[Lemma~3.1]{Jed92}}]
\label{corredif}
For any array $\varphi: A\rightarrow \{\pm 1\}$, 
\[
R_\varphi(x)= |A|+ 4(d_{\varphi}(x)
-|N_{\varphi}|)
\]
where $N_{\varphi}=\{a\in A \,|\,  
\varphi(a)=-1\}$ and 
$d_{\varphi}(x)=
|N_{\varphi}\cap (x+N_{\varphi})|$.
\end{lemma}
\begin{proof}
Routine counting. \hfill $\Box$
\end{proof}

\begin{theorem}
\label{th-goba-rqds}
Let $\phi$ be a binary ${\bf s}$-array, 
${\bf z}$ be a non-zero type vector, and 
$D=\{g+K\in E/K \mid  \phi'(g)=-1\}$.
Then $\phi$ is a \mbox{GOBA}$({\bf s})$ of type 
${\bf z}$ if and only if $D$ is a relative 
$(4t+2,2,4t+2,2t+1)$-quasi-difference set in $E/K$ 
with forbidden subgroup $H/K$; furthermore, $D$ is
extremal if $z_1=0$.
\end{theorem}
\begin{proof}
We continue with the notation of 
Lemma~\ref{corredif}.
By Lemma~\ref{GammaTransmit},
$D$ is a full transversal for $H/K$ in $E/K$.
Also, $|N_{\phi'}|=\allowbreak |E|/2$
by Lemma~\ref{PhidashFacts}; 
thus $|D| = \allowbreak |N_{\phi'}|/|K|$. 

For each $g\not \in H$, denote
$|D\cap (g+K+D)|$ by $d_{D}(g+ K)$: this is 
the number of $x+K\in D$ such that 
$x-g+K\in D$.
Since $d_{D}(g+ K)= d_{\phi'}(g)/|K|$,
Lemma~\ref{corredif} implies that
\begin{equation}\label{equivryd}
\begin{array}{ccl}
R_{\phi'}( g )=-2|H| \, & \
\Leftrightarrow \ & \ d_{D}( g+K)=2t\\
R_{\phi'}( g )=0 \, & \
\Leftrightarrow \ & \ d_{D}( g+K )=2t+1\\
 R_{\phi'}( g )=2|H| \, & \
 \Leftrightarrow \  & \ d_{D}( g+ K )=2t+2.
\end{array}
\end{equation}
Let $S=\{g+K\in D \mid  R_{\phi'}(g)=0\}$.
According to (\ref{equivryd}), 
Definition~\ref{defgoba}~(i) holds if and only if
\[
d_{D}(g+K)=
\left\{\begin{array}{ll}
2t+1 & \quad  g+K\in S +H/K\\
2t\ \  {\rm or}\ \  2t+2 & \quad 
\mbox{otherwise}.
\end{array}\right.
\]
Lemma~\ref{PhidashFacts} yields
\[
|S|=\frac{|\{g+K\in E/K\ | \  
R_{\phi'}(g)=0\}|}{2}=
|R_{\phi'}^{-1}(0)|/2|K|.
\]
Thus $|S| = 2t+1$ for $z_1=1$ if and only if  
Definition~\ref{defgoba}~(ii) holds.
\hfill $\Box$
\end{proof}
\begin{remark}\label{DandCompD}
Theorem~\ref{th-goba-rqds} remains valid 
when $D$ is replaced by its complement 
$\{g+K\in E/K \; | \; \allowbreak  \phi'(g)=
\allowbreak 1\}$.
\end{remark}
\begin{theorem}
\label{qc-goba} 
A (normalized) binary ${\bf s}$-array $\phi$ is 
a GOBA$({\bf s})$ 
of type ${\bf z}\not={\bf 0}$ if and only if 
$f_{\bf z}\partial\phi$ is quasi-orthogonal.
\end{theorem}
\begin{proof}
This is a consequence of Theorem~\ref{th-goba-rqds},
Remark~\ref{DandCompD}, 
Proposition~\ref{quasiort-quasidiff}, 
and Lemma~\ref{GammaTransmit}.  
\hfill $\Box$
\end{proof}  

We proceed to formulate `base' cases of 
Theorems~\ref{th-goba-rqds} and \ref{qc-goba}. 
Let $\partial \phi\in B^2(G,\mathbb{Z}_2)$.
Since $M_{\partial \phi}$ is Hadamard 
equivalent to a group-developed matrix, and 
such a matrix has constant row sum, 
$\partial \phi$ can be orthogonal only if
$|G|$ is square. 
This situation has been extensively studied.
\begin{theorem}
\label{BaseCase}
Let $|G|= 4u^2$, and let
$D$ be a subset of $G$ of size $2u^2-\allowbreak u$. 
Define $R=\{(\phi(g),g)\, |\,  g\in G\}\subset
{\mathbb{Z}}_2\times G$ where 
$\phi: G\rightarrow \{\pm 1\}$ is the characteristic 
function of $D$. Then the following 
are equivalent.
\begin{itemize}
\item[{\em (i)}] $\partial \phi$ is orthogonal.
\item[{\em (ii)}] $D$ 
is a Menon-Hadamard difference set in $G$.
\item[{\em (iii)}] $R$ 
is a relative $(4u^2,2,4u^2,2u^2)$-difference set 
in ${\mathbb{Z}}_2\times G$ with forbidden subgroup  
${\mathbb{Z}}_2\times \{1_G\}$.
\item[{\em (iv)}] $\phi$
is a perfect nonlinear function.
\end{itemize}
If $G$ is abelian then {\rm (i)} -- {\rm (iv)} 
are further equivalent to
\begin{itemize}
\item[{\em (v)}] $\phi$ is a PBA.
\end{itemize}
\end{theorem}
\begin{proof}
See \cite[Theorem 1]{Pot04} for (iii) 
$\Leftrightarrow$ (iv). The other equivalences
are given by Theorem~2.6 and Lemma~2.10 of 
\cite{DFH00}. \hfill $\Box$
\end{proof}
\begin{remark}
In Theorem~\ref{BaseCase} and Theorem~\ref{QOBaseCase} 
below we may assume that 
$\phi$ is normalized, by taking the complement of 
$D$ (and thus also of $R$) if necessary. 
\end{remark}

The next theorem is an analog of the previous one 
for $|G|\equiv 2 \mod 4$ (recall that we have not 
found quasi-orthogonal coboundaries over non-abelian 
$G$ at orders greater than $6$).
\begin{theorem}\label{QOBaseCase}
Let $G$ be abelian of order $4t+2$, and 
let $D$ be a $k$-subset of $G$  
with characteristic function 
$\chi:G\rightarrow \mathrm{GF}(2)$. Define 
$R=\{(\phi(g),g)\, |\, g\in G\}\subset 
{\mathbb{Z}}_2\times G$ where
$\phi(x)=(-1)^{\chi(x)}$.
Then the following are equivalent.
\begin{itemize}
\item[{\em (i)}] $\partial \phi$ is quasi-orthogonal.
\item[{\em (ii)}] $D$ is a 
$(4t+2,k,k-(t+1), (4t+1)(k-t)-k(k-1))$-almost 
difference set  in $G$.
\item[{\em (iii)}] 
$R$ is an extremal relative 
$(4t+2,2,4t+2,2t+1)$-quasi-difference 
set in ${\mathbb{Z}}_2\times G$ with forbidden 
subgroup $\mathbb{Z}_2\times \{1_G\}$. 
\item[{\em (iv)}] $\phi$ is an OBA. 
\end{itemize}
If a difference set with parameters
$(n,\frac{n\pm \sqrt{3n-2}}{2},
\frac{n+2\pm 2\sqrt{3n-2}}{4})$ does not exist, then 
{\em (i)} -- {\em (iv)} are further equivalent to
\begin{itemize}
\item[{\rm (v)}] 
$\chi$ has optimal  nonlinearity $(t+1)/(2t+1)$.
\end{itemize}
\end{theorem}
\begin{proof} 
Put $|G|=n$.

(i) $\Leftrightarrow$ (iv): 
Lemma~\ref{lemmaquasiortho} and the  fact that  
$\phi(g) R_\phi(g)$
is the sum of row $g$ in $M_{\partial \phi}$.

(i) $\Leftrightarrow$ (ii):
by Lemma~\ref{corredif}, $R_\phi(g)=2$ or $-2$ 
if and only if $d_{\phi}(g)=\allowbreak k-t-1$ 
or $k-t$, respectively.
Identity (19) of \cite{CD04} then accounts
for this part.

(i) $\Leftrightarrow$ (iii): 
Proposition \ref{quasiort-quasidiff} together
with the isomorphism $E_{\partial \phi} 
\rightarrow \mathbb{Z}_2\times G$ defined by 
$(u,g)\mapsto (u\phi(g),g)$;
cf.~Proposition~\ref{GammaRef}.

(ii) $\Leftrightarrow$ (v):
see \cite[Theorem~25]{CD04}. 
\hfill $\Box$
\end{proof}
\begin{remark}
The condition attached to (v) is only needed 
for (v) $\Rightarrow$ (ii). No difference 
sets with the stated parameters are known; 
see \cite[Remark~II, p.~224]{CD04}.
\end{remark}

We end this section with a discussion of
calculating GOBAs.
Label the elements of $G$ as  
$g_1=0,g_2,\ldots,g_{4t+2}$, and let 
$\delta_k\colon G\rightarrow \{\pm 1\}$
be the characteristic function of $\{ g_k\}$.
Up to relabeling, 
$\{\partial_2,\ldots,\partial_{4t+1}\}$ 
is a basis of $B^2(G,\langle -1\rangle)$, 
where $\partial_k:=\partial \delta_k$ is an
{\em elementary coboundary}.
Choose ${\bf z} \neq {\bf 0}$. We first  
try to find quasi-orthogonal 
$\psi\in Z^2(G,\langle -1\rangle)$
such that $f_{\bf z}\psi\in 
B^2(G, \langle -1\rangle)$.
Straightforward linear algebra  
gives the decomposition $\psi=   
f_{\bf z} \prod_k \partial_k^{i_k}$. 
Then $\phi = \prod_k \delta_k^{i_k}$ is 
a GOBA({\bf s}) of type ${\bf z}$ 
over $G$.
\begin{example}\label{ExampleOne}
The maps  
$\phi_1={\small 
\left[\begin{array}{rrr}
1 & -1 & \phantom{-} 1 \\
1 &  1 & \phantom{-} 1
\end{array}
\right]}$,
$\phi_2={\small 
\left[\begin{array}{rrr}
1 & \phantom{-} 1 &  -1\\
1 & \phantom{-} 1 &  1
\end{array}
\right]}$,
$\phi_3={\small 
\left[\begin{array}{rrr}
1 & 1 & -1 \\
1 &  -1 & \phantom{-} 1
\end{array}
\right]}$  
on ${\mathbb Z}_6={\mathbb Z}_2\times 
{\mathbb Z}_3$ are GOBA($2,3$)s of 
type ${\bf z}_1 = (1,0)$, 
${\bf z}_2 = (0,1)$,
${\bf z}_3 = (1,1)$, respectively. 
We display each quasi-orthogonal cocycle 
$f_{{\bf z}_i}\partial \phi_i$ 
as a Hadamard (componentwise) product:
\[
{\footnotesize \left[ \vspace{1pt}
\renewcommand{\arraycolsep}{.08cm}
\begin{array}{lrrrrr}
1 & \phantom{-}1 & \phantom{-}1 & 1 & 1 & 1  \\
1 & \phantom{-}1 & \phantom{-}1 & 1 & 1 & 1  \\
1 & \phantom{-}1 & \phantom{-}1 & 1 & 1 & 1  \\
1 & \phantom{-}1 & \phantom{-}1 & -1 & -1 & -1  \\
1 & \phantom{-}1 & \phantom{-}1 & -1 & -1 & -1  \\
1 & \phantom{-}1 & \phantom{-}1 & -1 & -1 & -1  
\end{array}
\vspace{1pt} \right] } \circ
{\footnotesize\left[ \vspace{1pt}
\renewcommand{\arraycolsep}{.08cm}
\begin{array}{rrrrrr}
1 & 1 & 1 & 1 & 1 & 1  \\
1 & 1 & -1 & -1 & -1 & -1  \\
1 & -1 & -1 & 1 & 1 & 1  \\
1 & -1 & 1 & 1 & -1 & 1  \\
1 & -1 & 1 & -1 & 1 & 1  \\
1 & -1 & 1 & 1 & 1 & -1  
\end{array}
\vspace{1pt} \right] } 
= {\footnotesize \left[ \vspace{1pt}
\renewcommand{\arraycolsep}{.08cm}
\begin{array}{rrrrrr}
1 & 1 & 1 & 1 & 1 & 1  \\
1 & 1 & -1 & -1 & -1 & -1  \\
1 & -1 & -1 & 1 & 1 & 1  \\
1 & -1 & 1 & -1 & 1 & -1  \\
1 & -1 & 1 & 1 & -1 & -1  \\
1 & -1 & 1 & -1 & -1 & 1  
\end{array}
\vspace{1pt} \right] }, 
\] 
\[
{\footnotesize \left[ \vspace{1pt}
\renewcommand{\arraycolsep}{.08cm}
\begin{array}{lrrrrr}
1 & 1 & 1 & \phantom{-}1 & 1 & 1   \\
1 & 1 & -1  & \phantom{-}1 & 1 & -1   \\
1 & -1 & -1 & \phantom{-}1 & -1 & -1   \\
1 & 1 & 1 & \phantom{-}1 & 1 & 1   \\
1 & 1 & -1 &\phantom{-}1 & 1 & -1   \\
1 & -1 & -1 & \phantom{-}1 & -1 & -1   
\end{array}
\vspace{1pt} \right] } \circ
{\footnotesize\left[ \vspace{1pt}
\renewcommand{\arraycolsep}{.08cm}
\begin{array}{rrrrrr}
1 & 1 & 1 & 1 & 1 & 1  \\
1 & -1 & -1 & 1 & 1 & 1  \\
1 & -1 & 1 & -1 & -1 & -1  \\
1 & 1 & -1 & 1 & 1 & -1  \\
1 & 1 & -1 & 1 & -1 & 1  \\
1 & 1 & -1 & -1 & 1 & 1  
\end{array}
\vspace{1pt} \right] } 
= {\footnotesize \left[ \vspace{1pt}
\renewcommand{\arraycolsep}{.08cm}
\begin{array}{rrrrrr}
1 & 1 & 1 & 1 & 1 & 1  \\
1 & -1 & 1 & 1 & 1 & -1  \\
1 & 1 & -1 & -1 & 1 & 1  \\
1 & 1 & -1 & 1 & 1 & -1  \\
1 & 1 & 1 & 1 & -1 & -1  \\
1 & -1 & 1 & -1 & -1 & -1  
\end{array}
\vspace{1pt} \right] }, 
\] 
\[
{\footnotesize \left[ \vspace{1pt}
\renewcommand{\arraycolsep}{.08cm}
\begin{array}{lrrrrr}
1 & 1 & 1 & \phantom{-}1 & 1 & 1   \\
1 & 1 & -1  & \phantom{-}1 & 1 & -1   \\
1 & -1 & -1 & \phantom{-}1 & -1 & -1   \\
1 & 1 & 1 & -1 & -1 & -1   \\
1 & 1 & -1 &-1 & -1 & 1   \\
1 & -1 & -1 & -1 & 1 & 1  
\end{array}
\vspace{1pt} \right] } \circ
{\footnotesize\left[ \vspace{1pt}
\renewcommand{\arraycolsep}{.08cm}
\begin{array}{rrrrrr}
1 & 1 & 1 & 1 & 1 & 1  \\
1 & -1 & -1 & -1 & -1 & 1  \\
1 & -1 & 1 & -1 & 1 & 1  \\
1 & -1 & -1 & 1 & -1 & -1  \\
1 & -1 & 1 & -1 & -1 & -1  \\
1 & 1 & 1 & -1 & -1 & 1  
\end{array}
\vspace{1pt} \right] } 
= {\footnotesize \left[ \vspace{1pt}
\renewcommand{\arraycolsep}{.08cm}
\begin{array}{rrrrrr}
1 & 1 & 1 & 1 & 1 & 1  \\
1 & -1 & 1 & -1 & -1 & -1  \\
1 & 1 & -1 & -1 & -1 & -1  \\
1 & -1 & -1 & -1 & 1 & 1  \\
1 & -1 & -1 & 1 & 1 & -1  \\
1 & -1 & -1 & 1 & -1 & 1  
\end{array}
\vspace{1pt} \right] }. 
\] 
Note that $f_{{\bf z}_2}\partial\phi_2$ is 
a quasi-orthogonal coboundary; as are all
the $\partial \phi_i$.
\end{example}
\begin{example}\label{ExampleOnetwoD}
The map 
${\small 
\left[\begin{array}{rrrrrrrrr}
1 & \ -1& 1 & \ -1 & \ -1 & \phantom{-}1 \\
1 & 1 & \ -1 & 1 & 1 & 1 \\
1 & 1 & 1 & 1 & 1 & 1
\end{array}
\right]}^\top$ 
on ${\mathbb Z}_6 \times {\mathbb Z}_3 = 
{\mathbb Z}_2 \times {\mathbb Z}_3
\times{\mathbb Z}_3$ 
is a GOBA($6,3$) of type ${\bf z} =(1,0)$. 
Its quasi-orthogonal cocycle is 
$f_{\bf z}\partial_4\partial_8
\partial_{10}\partial_{13}$. 
\end{example}

\section{Negaperiodic Golay pairs}
\label{Section4}

In this section we explore how GOBSs can be used 
to construct NGPs.
\begin{proposition}[{\cite[Theorem~3]{Ega16}}] 
\label{equivNGP-RDS}
Binary sequences $\phi_1$, $\phi_2$ of length 
$2t$ form an NGP if and only if  
$\{x^i \mid \phi_1'(i)=1\}\cup
\{x^iy \mid  \phi_2'(i)=1\}$
is a relative $(4t,2,4t,2t)$-difference set
in the dicyclic group
$Q_{8t}=\langle x,y \,|\, 
x^{2t}=y^2,\allowbreak y^4=1,\,
y^{-1}xy=x^{-1}\rangle$.
\end{proposition}

\begin{remark} \label{Hadgroups} 
By Proposition~\ref{equivNGP-RDS} and
\cite[Theorems~5.6 and 5.7]{ACP01}, 
NGPs of length $(q+1)/2$ exist for all prime 
powers $q\equiv  3 \ \mod 4$. 
\end{remark}

Proposition~\ref{equivNGP-RDS} ties NGPs
into the mainstream theory of cocyclic
Hadamard matrices: 
by \cite[Proposition~6.5]{Fla97},
existence of a $(4t,2,4t,2t)$-difference set
in $Q_{8t}$ is equivalent to existence of certain 
orthogonal cocycles over the dihedral group 
$D_{4t}$ of order $4t$. (Incidentally, this 
gives another justification of Remark~\ref{Hadgroups}, 
via Ito's Hadamard groups of 
quadratic residue type~\cite[pp.~986--987]{Ito94}.)
These cocycles lie in a single cohomology class,
with representative labeled $(A,B,K)=(1,-1,-1)$
in \cite{Fla97}; $A$, $B$ are `inflation' variables 
and $K$ is the `transgression' variable in a 
Universal Coefficients theorem decomposition of 
$H^2(D_{4t},\mathbb{Z}_2)$.

The next theorem makes Proposition~\ref{equivNGP-RDS}
more explicit. It shows how to translate directly 
between cocycles and NGPs.
When the latter are complementary GOBSs, this 
implies existence of orthogonal cocycles if
there exist quasi-orthogonal cocycles at half 
the order (unfortunately, the process does 
not reverse). 
\begin{theorem}\label{ExplicitProp5}
Let $G=\langle a, b \, |\,  a^{n}=b^2=1, a^b 
= a^{-1}\rangle\cong D_{2n}$
with elements ordered as
$1, a, \ldots , a^{n-1}, b, \allowbreak 
ab, \ldots , a^{n-1}b$.
Also let $\phi_1$, $\phi_2$ be binary sequences
of length $n$, and define $j_{k,i}$ to be $1$ or $0$ 
depending on whether
$\phi_i(k) = -1$ or $1$, respectively. 
Then $(\phi_1,\phi_2)$ is an NGP
if and only if
$\lambda \prod_{k=1}^n
\partial_{k}^{j_{k,1}}\partial_{n+k}^{j_{k,2}}$
is an orthogonal cocycle over $G$, where 
$\lambda$ is the cohomology class representative 
labeled $(A,B,K)= (1,-1,-1)$ in 
{\em \cite[Section~6]{Fla97}}.
\end{theorem}
\begin{proof} 
The center of
$\langle x,y \,|\, x^{n}=y^2,\allowbreak
y^4=1,\, y^{-1}xy=x^{-1}\rangle \cong Q_{4n}$ 
is $\langle x^n\rangle$.
Since $G\cong Q_{4n}/\langle x^n\rangle$, we may 
define a transversal map
$\sigma: G \rightarrow Q_{4n}$ by
\[
a^i \mapsto x^{i+n\delta_{\scst \phi_1(i),-1}}, 
\qquad
a^ib  \mapsto x^{i+n\delta_{\scst \phi_2(i),-1}}y
\]
where $\delta$ is the Kronecker delta.
Assuming that $\phi_1$ and $\phi_2$ are normalized,
let $\psi$ be the cocycle for $\sigma$, i.e.,
$\psi(g,h) = \sigma(g)\sigma(h)\sigma(gh)^{-1}$. 
By Proposition~\ref{equivNGP-RDS} and 
\cite[Corollary~2.5]{DFH00}, $\psi$ is orthogonal
if and only if $(\phi_1,\phi_2)$ is an NGP.

Set $\varphi(a^i) = \phi_1(i)$ and 
$\varphi(a^ib) = \phi_2(i)$.
Then $\lambda = \psi\hspace{.5pt}
\partial\varphi$ has matrix
\[\small 
\left[
\renewcommand{\arraycolsep}{.075cm}
\begin{array}{cr}
A & A\\
B & -B
\end{array}\right]
\]
where 
$A= 
[(-1)^{\lfloor (i+j)/n
\rfloor}]_{0\leq i,j\leq n-1}$ 
is back negacyclic, and $B$ is $A$ with rows 
$r$ and $n-r+1$ swapped for $1\leq r \leq n$.
Furthermore, 
$\partial \varphi = \prod_{k=1}^n
\partial_{k}^{j_{k,1}}\partial_{n+k}^{j_{k,2}}$
under the stipulated ordering of $G$.
\hfill $\Box$
\end{proof}
 
We now undertake a case study of quasi-orthogonal 
cocycles over cyclic groups.
Let $G=\mathbb{Z}_{4t+2}$ and index matrices by $1,\ldots,4t+2$ in this order. 
The set ${\cal B}=\{\gamma, \partial_i\, |\, \allowbreak 
2\leq i \leq 4t+2\}$ where $\gamma =\gamma_{4t+2}$
(as defined before Proposition~\ref{Proposition3})
is a basis of $Z^2(G,\mathbb{Z}_2)$.
We get an elementary coboundary matrix 
$M_i:= M_{\partial_i}$ by normalizing the back 
circulant matrix whose first row is $1$s 
except for the $i$th entry.
Also, $M_\gamma$ is the back 
negacyclic matrix $N$ of order $4t+2$. 
\begin{lemma}\label{propositiontquenacy}
Let $\psi \in Z^2(G,\mathbb{Z}_2)
\setminus B^2(G,\mathbb{Z}_2)$, say 
$M_\psi=M_{i_1} \circ \cdots 
\circ M_{i_w} \circ N$. Then
\begin{itemize}
\item[{\rm (i)}] 
up to sign, $M_\psi$ has $i$th row sum equal 
to its $(4t+4-i)$th row sum.
\item[{\rm (ii)}] The $(2t+2)$th  row sum of 
$M_\psi$ is $0$.
\item[{\rm (iii)}] $\psi$ is quasi-orthogonal 
if and only if the $i$th row sum of $M_\psi$ is 
$0$ for even $i$ and $\pm 2$ for odd $i>1$.
\end{itemize}
\end{lemma}
\begin{proof} 
If $\psi \in \mathcal B$ then row $i>2t+2$ of 
$M$ or its negation is 
row $(4t+4-i)$ cycled $4t+4-i-1$ positions to 
the right. Part~(i) then follows.
For (ii), observe that row $2t+2$ in $N$ is 
$[1\stackrel{2t+1}{\cdots}1$ 
 $-1\stackrel{2t+1}{\cdots}-1]$,
whereas the first half of row $2t+2$ in $M_i$ is 
identical to the second half.
Finally, (iii) holds
because the number of $-1$s in any row of $M_i$ 
is even; and the rows of $N$ indexed by an even 
(respectively, odd) integer have an odd 
(respectively, even) number of $-1$s.
\hfill $\Box$
\end{proof}

We use an approach borrowed from \cite{AAFR08} 
to count the negative entries in a $G$-cocyclic
matrix. Negating row $i$ of $M_{i}$ 
gives a {\em generalized coboundary matrix} 
$\overline{M}_{i}$, with exactly two $-1$s 
in each non-initial row $r$: these are in
columns $i$ and $[i-r+1]_{4t+2}$, 
where $[m]_n\in \{ 1, \ldots , n\}$ denotes 
the residue of $m$ modulo $n$.  
(Although $\overline{M}_i$ is not cocyclic, row 
negation preserves row excess.)
Hence the two generalized coboundary 
matrices with $-1$ in position $(r,c)$
are $\overline{M}_c$ and 
$\overline{M}_{[r+c-1]_{4t+2}}$.

A set $\{\overline{M}_{i_j}\colon 1 \leq j \leq w\}$  
defines an {\em  $r$-walk} if there is an 
ordering  
$\overline{M}_{l_1} ,\ldots, \overline{M}_{l_w}$ 
of its elements such that $\overline{M}_{l_i}$ and 
$\overline{M}_{l_{i+1}}$ both have $-1$ in
row $r$ and column $l_{i+1}$, for $1\leq i\leq w$.
The walk is an $r$-{\em path} if its initial 
(equivalently, final) element shares a $-1$ in row 
$r$ with a generalized coboundary matrix not in the 
walk itself.
Clearly, the number of $-1$s in row $r$ of 
$\overline{M}_{i_1} \circ \cdots \circ 
\overline{M}_{i_w}$ is $2 {\cal C}_r$  where
${\cal C}_r$ is the number of maximal $r$-paths 
in $\{\overline{M}_{i_1} ,\ldots, 
\overline{M}_{i_w}\}$.
To calculate $\mathcal{C}_r$ we set up a bipartite 
graph on vertex sets $S = \{i_1,\ldots,i_w\}$ 
and $T= \{[i_1-r+1]_{4t+2},\ldots,
[i_w-r+1]_{4t+2}\}$. Draw an edge between $i_j\in S$ 
and $l \in T$ if $i_j = l$ or 
$l=[i_j-r+1]_{4t+2}\in S$. 
The number of maximal paths in this
bipartite graph is $\mathcal{C}_r$.

Next, let ${\cal I}_r$ be the number of columns where 
$N$ and $\overline{M}_{i_1} \circ
\cdots \circ 
\overline{M}_{i_w}$ share a
$-1$ in row $r$. These column indices comprise 
the intersection of $\{ 4t+4-r,\ldots ,4t+2\}$ and the 
set of endpoints of the previously calculated maximal 
$r$-paths.
\begin{theorem}[cf.{~\cite[Proposition 1]{AAFR08}}]
\label{FinalCount}
A $\mathbb{Z}_{4t+2}$-cocyclic matrix 
$M_{i_1} \circ \cdots \circ 
M_{i_w} \circ N$ 
is quasi-orthogonal if and only if, for 
$2\leq r\leq 2t+1$,
\[
\begin{array}{ll}
{\cal C}_r \in\{{\cal I}_r+t+\frac{1-r}{2},\,
{\cal I}_r+t+\frac{3-r}{2}\}
& \qquad r\,\,\mbox{odd} \\
\vspace{-12pt} & \\
{\cal C}_r ={\cal I}_r+t+1-\frac{r}{2} 
& \qquad r\,\,\mbox{even}.
\end{array}
\]
\end{theorem}
\begin{proof}
The number of $-1$s in row $r$ of 
$\overline{M}_{i_1} \circ \cdots 
 \circ 
\overline{M}_{i_w} \circ N$ is 
$2 {\cal C}_r + r-1 - 2{\cal I}_r$, so
Lemma~\ref{propositiontquenacy} gives the result.
\hfill $\Box$
\end{proof}
\begin{corollary}\label{tBound}
Let $\psi=\gamma\prod_{j=1}^w \partial_{i_j}$ 
with $\partial_{i_j}\in{\cal B}$. 
If $\psi$ is quasi-orthogonal then 
$t\leq w\leq 3t+1$.
\end{corollary}
\begin{proof}
We have ${\cal I}_2=0$, and ${\cal C}_2 =t$
by Theorem~\ref{FinalCount}.
Thus $t\leq w$. 
On the other hand, since 
the basis of coboundaries forms a $2$-path, 
at least $t-1$ coboundaries must be removed to 
get $t$ $2$-paths. Hence $w\leq 4t-(t-1)$.
\hfill $\Box$
\end{proof}

Corollary~\ref{tBound} is equivalent to
\begin{lemma}
If $\phi\colon \mathbb{Z}_{4t+2} \rightarrow 
\{\pm 1\}$ is a GOBS containing $w$
occurrences of $-1$ then $t\leq w \leq 3t+1$. 
\end{lemma}
\begin{proof}
Negating  all odd index entries or all even 
index entries of a GOBS produces another GOBS. 
So it may be assumed that $\phi(0)=\phi(4t+1)=1$.
\hfill $\Box$
\end{proof}

We search for NGPs in the set of quasi-orthogonal 
cocycles over ${\mathbb Z}_{4t+2}$, motivated
by the ubiquity of these cocycles and the 
optimal autocorrelation of each map in the 
resulting pair.
Computer-aided searches 
found the NGPs in Table~1. 
\begin{center}\begin{table}[htb]
\label{tabla1}
\[
\begin{array}{|c||c|c|} \hline k  
& \phi_1 & \phi_2 \\ \hline
\hline  3  &  1^2,4  &  2,1,3   \\
\hline 5 & 2,1^3,5 & 3,1,2,1,3   \\
 \hline 7 & 
2,1,5,1^3,3 & 2,1,4,2,1^2,3 \\
\hline
9 & 3,1,2,1^3,3,1,5
& 2,1,2,3,2,1^3,5 \\
\hline
13 & 3,3,2,2,1,2,1,2,1^4,6 
& 3,3,1,3,1,2,1,2,1^4,6 \\
\hline 
15 & \ 3,2,4,1^2,2,2,1,2,1^5,7 \ 
& \ 3,2,3,2,1,2,2,1,2,1^5,7 \ \\
\hline
\end{array}
\]
\caption{NGPs $(\phi_1, \phi_2)$ from 
quasi-orthogonal cocycles over 
$\mathbb{Z}_{2k}$}
\end{table}
\end{center}
\noindent 
Each sequence in Table~1 starts with $1$ and is 
designated by an integer string,
where $i$ in the string means a run of $i$ 
identical entries in the sequence, 
and $1^j$ is an alternating 
subsequence of length $j$.
There are no NGPs among the sequences coming 
from quasi-orthogonal
cocycles over $\mathbb{Z}_{22}$
(however, as we know, NGPs of length $22$ exist).
This gap could be related to the maximal 
determinant problem: the Ehlich-Wojtas bound is 
not attainable because 
$21$ is not a sum of two squares. 

Egan~\cite{Ega16} classified NGPs of length 
$2k$ for $k\leq 10$ up to equivalence 
with respect to five elementary operations as
defined in \cite{BD15}. The set of NGPs that come
from GOBSs is invariant under each elementary 
operation. Table~2 records the 
number $\hat{n}(k)$ of such NGPs of length $2k$, 
and the number $\hat{d}(k)$ of their equivalence 
classes.
To compare against 
\cite[Table 2]{Ega16}, we have included the 
total number $n(k)$ of NGPs of length $2k$ and
the number $d(k)$ of their equivalence 
classes. 
\begin{center}\begin{table}[h]\label{tabla2}
\[
\begin{array}{|c||c|c|c|c|} \hline 
\  k \ & n(k) & \hat{n}(k) & \ d(k) \
& \ \hat{d}(k) \ \\
\hline \hline
3  &  576  &  576  &  1  &  1   \\
\hline 
5  & 11 200  & 4 800  &  3  &  2  \\
\hline  
7  &  90 944  &  18 816  &  5 & 1 \\ 
\hline
9 & \ 1 041 984 \ & \ 62 208 \ &  20  &  2  
\\
\hline

\end{array}
\]
\caption{Enumeration of NGPs and
their equivalence classes}
\end{table}
\end{center}
 \section{Normal cocyclic matrices} 
\label{NormalMatrices}
This section is essentially independent of the 
main thrust of the paper. Nonetheless, it  
addresses a fundamental question in 
algebraic design theory, which we answer 
in special cases that were the focus of 
Section~\ref{Section4}.

A matrix $M$ is \textit{normal} if it 
commutes with its transpose (possibly up to 
row or column permutations), i.e.,
$\mathrm{Gr}(M)=\mathrm{Gr}(M^\top)$, where
$\mathrm{Gr}(M)$ denotes the 
Grammian $MM^\top$. Many kinds of pairwise 
combinatorial designs are normal matrices
(the defining pairwise constraint on rows 
implies the same constraint on columns; 
see \cite[Chapter~7]{DF11}). 
We also note that the matrix of a quasi-orthogonal 
cocycle is normal~\cite[Remark~6]{AF17}. 
Thus, by the following lemma derived from
(\ref{condiciondecociclo}), 
a cocycle $\psi$ is quasi-orthogonal if and 
only if $M_\psi$ has optimal column excess. 
\begin{lemma}\label{REeqCE}
For any group $G$ and 
$\psi \in Z^2(G,\mathbb Z_2)$,
\[
\mathrm{Gr}(M_\psi)_{ij}=
\psi(g_ig_j^{-1},g_j)
\sum_{g\in G}\psi(g_ig_j^{-1},g)
\]
and
\[
\mathrm{Gr}(M_\psi^\top)_{ij}=
\psi(g_i,g_i^{-1}g_j)
\sum_{g\in G}\psi(g,g_i^{-1}g_j).
\]
\end{lemma}

We use Lemma~\ref{REeqCE} to prove that cocyclic
matrices for two familiar classes of indexing 
groups are normal.
\begin{proposition}
Let $G$ be abelian or dihedral of order $2m$, 
$m$ odd, and let $\psi\in Z^2(G,\mathbb{Z}_2)$ 
where $\psi \not \in B^2(G,\mathbb{Z}_2)$ 
if $G$ is dihedral.  
Then $M_\psi$ is normal (under the same indexing 
of rows and columns by the elements of $G$).
\end{proposition}
\begin{proof}
We suppose that $G$ is generated by $a$ and $b$,
with $a^m = b^2=1$, and index rows and 
columns by the elements of $G$ under the 
ordering $1,a,\ldots,a^{m-1},\allowbreak 
 b,ab,\ldots,a^{m-1}b$.
A representative $\beta$ for the 
non-identity element of $H^2(G,\mathbb{Z}_2)$
has matrix 
\[
\small{
\left[ 
\begin{array}{rr}
\renewcommand{\arraycolsep}{.075cm}
 J & J \\ 
 J & -J 
 \end{array}\right]}.
 \]
Thus, if $G$ is abelian then $M_\psi$ is 
symmetric and so trivially normal.

Henceforth $G$ is dihedral.
Let $\psi = \beta \partial \phi$.
We collect together some basic 
properties of $M_\psi$.
\begin{enumerate}
\item[(i)]
For each $i$, 
$\{ \partial \phi(a^ib,a^j) \; |\; 
1\leq j \leq m\} = 
\{ \partial \phi(a^ib,a^jb) \; |\; 
1\leq j \leq m\}$; and for each $j$,
$\{ \partial \phi(a^i,a^jb) \; |\;
1\leq i \leq m\}=
\{ \partial \phi(a^ib,a^jb) \; |\; 
1\leq i \leq m\}$.
Thus, if $k>m$ then the $k$th row sum and 
$k$th column sum of $M_\psi$ are zero. 
\item[(ii)]  
Since $\{ \partial \phi(a^i,a^jb) \; |\; 
1\leq j \leq m\} = 
\{ \partial \phi(a^jb,a^i) \; |\; 
1\leq j \leq m\}$, 
the $k$th row sum of $M_\psi$ equals 
its $k$th column sum for $k\leq m$.
\end{enumerate}

Now we consider the Grammian quadrants
in turn.

If $1\leq i\leq m$ and $m+1\leq j\leq 2m$ then
\[
\mathrm{Gr}(M_\psi)_{ij}=
\psi(a^{i+j-2}b,a^{j-1} b)
\sum_{g\in G}\psi(a^{i+j-2}b,g)=0
\]
by Lemma~\ref{REeqCE} and (i); 
$\mathrm{Gr}(M_\psi^\top)_{ij}=0$
similarly.

Let $1\leq i\leq m$ and $1\leq j\leq m$. Then
\[
\mathrm{Gr}(M_\psi)_{ij}=
\partial\phi(a^{i-j},a^{j-1} )
\sum_{g\in G}\partial \phi(a^{i-j},g) 
= \phi(a^{j-1})\phi(a^{i-1})
\sum_{g\in G}\phi(g)\phi(a^{i-j}g)
\]
and 
\[
\mathrm{Gr}(M_\psi^\top)_{ij}=
\phi(a^{j-1})\phi(a^{i-1})
\sum_{g\in G}\phi(g)\phi(ga^{j-i}).
\]
These entries are equal by the identity
$\sum_{k=1}^m\phi(a^k)\phi(a^{k+1})=
\allowbreak 
\sum_{k=1}^m\phi(a^k)\phi(a^{k-1})$.

Finally, let $m+1\leq i , j\leq 2m$.
Then
\[
\mathrm{Gr}(M_\psi)_{ij}=
\psi(a^{i-j},a^{j-1}b )
\sum_{g\in G}\psi(a^{i-j},g) 
\]
and 
\[
\mathrm{Gr}(M_\psi^\top)_{ij}=
\psi(a^{i-1}b,a^{i-j} )
\sum_{g\in G}\psi(g,a^{i-j}). 
\]
Since $\psi(a^{i-1}b,a^{i-j})=  
\partial\phi(a^{i-1}b,a^{i-j}) 
=\partial \phi (a^{i-j},a^{j-1}b)=
\psi(a^{i-j},a^{j-1}b )$,
we are done by (ii).
\hfill $\Box$
\end{proof}
\begin{remark}
There are plenty of examples of non-normal 
cocyclic matrices $M_\psi$ for 
$\psi\not\in B^2(G,\mathbb{Z}_2)$
and $|G|$ divisible by $4$.
\end{remark}

\subsubsection*{Acknowledgments.} 
The  authors thank Kristeen Cheng for 
reading the manuscript, and V\'ictor \'Alvarez 
for his assistance with computations. 
We are also grateful to Ronan Egan, who shared 
his insights on NGPs with us. 
Remark~\ref{Hadgroups}  and reference 
\cite{ACP01} were kindly provided by one of 
the referees. This research has been supported 
by project FQM-016 funded by JJAA (Spain).

\end{document}